\documentclass[a4paper, 12pt, reqno]{amsart}
\usepackage[left=1in, right=1in, top=1.2in, bottom=1.2in]{geometry}

\usepackage{times}
\usepackage{microtype}

\usepackage{commath,amssymb,amscd,mathrsfs,mathtools,dsfont,bbm}
\usepackage[all]{xy}
\usepackage{enumerate}

\usepackage{cite}       %won't show labels for citation

\usepackage[usenames,dvipsnames,svgnames,table]{xcolor}
\definecolor{red}{RGB}{255, 25, 25}
\definecolor{blue}{RGB}{25, 50, 200}

\def\MR#1{\href{http://www.ams.org/mathscinet-getitem?mr=#1}{MR#1}}
\usepackage[pagebackref, linktocpage]{hyperref}
\hypersetup{
    pdftitle={Jordan},            % title
    pdfauthor={Fei Hu},     % author
    colorlinks=true,        % false: boxed links; true: colored links
    linkcolor=red,          % color of internal links (change box color with linkbordercolor)
    citecolor=blue,         % color of links to bibliography
    filecolor=magenta,      % color of file links
    urlcolor=cyan           % color of external links
}

\newtheorem{theorem}{Theorem}[section]
\newtheorem{lemma}[theorem]{Lemma}

\theoremstyle{definition}
\newtheorem{definition}[theorem]{Definition}

\newtheorem{question}[theorem]{Question}

\theoremstyle{remark}
\newtheorem{remark}[theorem]{Remark}

\newtheorem{step}{\sc Step}

\numberwithin{equation}{section}

\renewcommand{\lhd}{\trianglelefteq}
\newcommand{\wtilde}[1]{\widetilde{#1}}
\newcommand{\wbar}[1]{\overline{#1}}

\newcommand{\lra}{\longrightarrow}

\newcommand{\comp}{c_G}

\newcommand{\bk}{\mathbbm{k}}

\newcommand{\bG}{G} %% now use normal G to denote an algebraic group

\newcommand{\bF}{\mathbb{F}}

\newcommand{\red}{{\rm red}}

\newcommand{\aff}{{\rm aff}}
\newcommand{\ant}{{\rm ant}}

\newcommand{\GL}{\mathrm{GL}}
\newcommand{\SL}{\mathrm{SL}}

\newcommand{\aut}{\mathbf{Aut}}
\newcommand{\Aut}{\operatorname{Aut}}

\newcommand{\Bir}{\operatorname{Bir}}

\newcommand{\ch}{\operatorname{char}}

\newcommand{\Cr}{\operatorname{Cr}}

\newcommand{\isom}{\simeq}

\newcommand{\rank}{\operatorname{rank}}

\begin{document}

\title[Jordan property for algebraic groups and automorphism groups]{Jordan property for algebraic groups and automorphism groups of projective varieties in arbitrary characteristic}

\author{Fei Hu}
\address{\textsc{Department of Mathematics, University of British Columbia, 1984 Mathematics Road, Vancouver, BC V6T 1Z2, Canada}
\endgraf \textsc{Pacific Institute for the Mathematical Sciences, 2207 Main Mall, Vancouver, BC V6T 1Z4, Canada}}
\email{\href{mailto:fhu@math.ubc.ca}{\tt fhu@math.ubc.ca}}

\begin{abstract}
We show an analogue of Jordan's theorem for algebraic groups defined over a field $\bk$ of arbitrary characteristic.
As a consequence, a Jordan-type property holds for the automorphism group of any projective variety over $\bk$.
\end{abstract}


%\date{\today}


\thanks{The author was partially supported by a UBC-PIMS Postdoctoral Fellowship.}

\maketitle

%\tableofcontents

%%%%%%%%%%%%%%%%%%%%%%%%%%%%%%%%%%%%%%%%%%%%%%%%%%%%%%%%%%

\section{Introduction} \label{section-intro}

%%%%%%%%%%%%%%%%%%%%%%%%%%%%%%%%%%%%%%%%%%%%%%%%%%%%%%%%%%

\noindent
In 1878, Camille Jordan \cite{Jordan1878} proved the following remarkable theorem.

\begin{theorem}[{cf.~\cite{Jordan1878}}] \label{theorem-Jordan}
For any positive integer $n$, there exists a constant $J(n)$ such that any finite subgroup $\Gamma$ of $\GL_n$ over a field of characteristic zero contains a normal abelian subgroup $A$ of index $\le J(n)$.
\end{theorem}

However, the above theorem is false for fields of characteristic $p > 0$ due to the existence of unipotent elements of finite order.
For instance, the group $\GL_n(\wbar \bF_p)$ contains arbitrarily large subgroups of the form $\SL_n(\bF_{p^r})$ which are simple modulo their centers.
Nevertheless, for any finite subgroup $\Gamma$ of $\GL_n(\bk)$ of order not divisible by $\ch(\bk)$, there still exists a normal abelian subgroup $A$ of $\Gamma$ with $[\Gamma : A] \le J(n)$ for the same $J(n)$ as in Theorem~\ref{theorem-Jordan} (see e.g. \cite[2.9]{BF66}).
Later, Serre showed that the Cremona group $\Cr_2(\bk)$ of rank $2$ over a field $\bk$ also has this property (cf.~\cite[Theorem~5.3]{Serre09}).
This motivates us to make the following definition.

\begin{definition} \label{def-p-Jordan}
Let $p$ a prime number or zero.
A group $G$ is called a {\it generalized $p$-Jordan group}, if there exists a constant $J(G)$, depending only on $G$, such that every finite subgroup $\Gamma$ of $G$ whose order is not divisible by $p$ contains a normal abelian subgroup $A$ of index $\le J(G)$.
We also call such a group {\it generalized Jordan} if there is no confusion caused (e.g., in practice, $p$ will always denote the characteristic of the ground field $\bk$).
\end{definition}

Note that when $p = 0$, this notion coincides with Popov's \cite[Definition~2.1]{Popov11}.
The above-mentioned results can be reformulated as follows.
Both general linear groups $\GL_n(\bk)$ and the Cremona group $\Cr_2(\bk)$ of rank $2$ are generalized $p$-Jordan, where $p = \ch(\bk)$.
Our first result below shows that in addition to above, any algebraic group is generalized Jordan.

\begin{theorem} \label{theorem-alg-gp}
Any algebraic group $G$ defined over a field $\bk$ of characteristic $p\ge 0$ is generalized $p$-Jordan.
Namely, there exists a constant $J(G)$, depending only on $G$, such that every finite subgroup $\Gamma$ of $G(\bk)$ whose order is not divisible by $p$ contains a normal abelian subgroup of index $\le J(G)$.
\end{theorem}

%An algebraic group $G$ over $\bk$ is a $\bk$-group scheme of finite type. For our purpose, we may always assume that $G$ is smooth (or equivalently, reduced), since we only consider finite subgroups of the group $G(\bk)$ of $\bk$-rational points.

\begin{remark}
Even in characteristic zero, Theorem~\ref{theorem-alg-gp} is nontrivial and has just been proved by Meng and Zhang recently (cf.~\cite[Theorem~1.3]{MZ18}).
The main obstruction is that the Jordan property may not be preserved under group extensions (see e.g. \cite[Remark~2.12]{Popov11} or \cite[Example~7]{Popov14}). See also \cite{Zarhin14} for a counterexample.
Note, however, that the argument in \cite{MZ18} also depends on the Levi decomposition of connected algebraic groups (cf.~\cite{Mostow56}), which is not available in prime characteristic (see e.g. \cite[Remark~4.9]{Brion15} and references therein).
Our approach is based on the investigation of the Jordan property of quotient groups (see e.g. Lemma~\ref{lemma-finite-quotient} and Step~\ref{alggp-step2} in the proof of Theorem~\ref{theorem-alg-gp}).
\end{remark}

Another generalization of Jordan's theorem to prime characteristic was due to Brauer and Feit \cite{BF66} by allowing arbitrary finite subgroup $\Gamma$ of $\GL_n(\bk)$ whose order may be divisible by $p > 0$.
They showed that $\Gamma$ contains a normal abelian subgroup whose index is bounded by a constant depending on $n$ as well as the order of the $p$-Sylow subgroup $\Gamma_{(p)}$ of $\Gamma$. 
Larsen and Pink \cite{LP11} has subsequently extended Brauer--Feit \cite{BF66} as follows.

\begin{theorem}[{cf.~\cite[Theorem~0.4]{LP11}}] \label{theorem-LP}
For any positive integer $n$, there exists a constant $J'(n)$ such that any finite subgroup $\Gamma$ of $\GL_n$ over a field $\bk$ of characteristic $p > 0$ contains a normal abelian $p'$-subgroup $A$ of index $\le J'(n) \cdot |{\Gamma_{(p)}}|^3$.
\end{theorem}

Here a finite group is called a {\it $p$-group} (resp. {\it $p'$-group}) if its order is some power of $p$ (resp. relatively prime to $p$).
In an analogous way, we introduce the following notion.

\begin{definition} \label{def-s-p-Jordan}
Let $p$ be a prime number.
A group $G$ is called a {\it $p$-Jordan group}, if there exist constants $J'(G)$ and $e(G)$, depending only on $G$, such that every finite subgroup $\Gamma$ of $G$ contains a normal abelian $p'$-subgroup $A$ of index $\le J'(G) \cdot |{\Gamma_{(p)}}|^{e(G)}$.
\end{definition}

Below is our second main result extending Theorem~\ref{theorem-LP} to arbitrary algebraic groups.

\begin{theorem} \label{theorem-alg-gp-LP}
Any algebraic group $G$ defined over a field $\bk$ of characteristic $p > 0$ is $p$-Jordan.
That is, there are constants $J'(G)$ and $e_G$, depending only on $G$, such that any finite subgroup $\Gamma$ of $G(\bk)$ contains a normal abelian $p'$-subgroup $A$ of index $\le J'(G) \cdot |{\Gamma_{(p)}}|^{e(G)}$.
\end{theorem}

\begin{remark}
We may think of Theorem~\ref{theorem-alg-gp-LP} as a stronger version of Theorem~\ref{theorem-alg-gp} in positive characteristic, if we only care about the existence of those constants (i.e., $J(G)$ and $J'(G)$).
Actually, we will see in Section \ref{sec-proof} that our constant $J(G) = c_GJ(n)^{c_G}$, where $\comp$ is the number of connected components of $G$ and $n$ is the least dimension of a faithful representation of $(\bG^{\circ})_\aff$ over $\bk$; see \S\ref{subsec-agtt} for the meaning of $(\bG^{\circ})_\aff$.
But $J'(G)$ is much more complicated and involved.
Also, it will be shown that $e(G) = 3(r_G + 1) \comp$, where $r_G$ is bounded by the rank of $(G^{\circ})_\aff$.
\end{remark}

In characteristic zero, it has also been proved by Meng and Zhang that the automorphism group $\Aut(X)$ is Jordan for any projective variety $X$ (cf.~\cite[Theorem~1.6]{MZ18}).
In this note, as a byproduct of our main theorems, we also deduce two Jordan-type properties for automorphism groups of projective varieties in arbitrary characteristic.

\begin{theorem} \label{thm-aut}
Let $X$ be a projective variety defined over a field $\bk$ of characteristic $p \ge 0$.
Then there exists a constant $J_X$, depending only on $X$, such that every finite $p'$-subgroup $\Gamma$ of $\Aut(X)$ contains a normal abelian subgroup $A$ of index $\le J_X$.
\end{theorem}

\begin{theorem} \label{thm-aut-LP}
Let $X$ be a projective variety defined over a field $\bk$ of characteristic $p > 0$.
Then there exist constants $J'_X$ and $e_X$, depending only on $X$, such that every finite subgroup $\Gamma$ of $\Aut(X)$ contains a normal abelian $p'$-subgroup $A$ of index $\le J'_X \cdot |{\Gamma_{(p)}}|^{e_X}$.
\end{theorem}

We also note that in characteristic zero, Prokhorov and Shramov proved that the group $\Bir(X)$ of birational self-maps of any non-uniruled variety $X$ is Jordan (cf.~\cite[Theorem~1.8(ii)]{PS14}). Assuming the Borisov--Alexeev--Borisov conjecture, which was recently proved in Birkar's pioneering work \cite{Birkar16}, they even showed in \cite{PS16} that $\Bir(X)$ is (uniformly) Jordan for any rationally connected variety $X$, generalizing Serre's \cite[Theorem~5.3]{Serre09} (the characteristic zero side).
Quite recently, as a consequence of the aforementioned Jordan property, Reichstein obtained new low bounds on the essential dimension of a series of finite groups which was not previously known even for special cases (cf.~\cite[Theorem~3]{Reichstein18}).
At the end of our introduction, we raise the following natural question (see also \cite[Question~6.1]{Serre09} for a related question).

\begin{question} \label{question-Serre}
Let $\Cr_n(\bk)$ be the Cremona group of rank $n\ge 2$ defined over a field $\bk$ of characteristic $p > 0$. Then is $\Cr_n(\bk)$ $p$-Jordan?
\end{question}

%\begin{remark} \label{rmk-periodic}
%Note that Jordan's theorem has a more general version for periodic linear groups (see e.g. \cite[Theorem~1.L.4]{KW73}). We believe that there should be a corresponding statement for periodic subgroups of algebraic groups.
%\end{remark}

%%%%%%%%%%%%%%%%%%%%%%%%%%%%%%%%%%%%%%%%%%%%%%%%%%%%%%%%%%

\section{Preliminaries} \label{sec-prelim}

%%%%%%%%%%%%%%%%%%%%%%%%%%%%%%%%%%%%%%%%%%%%%%%%%%%%%%%%%%

\subsection{Two group-theoretic lemmas} \label{subsec-gtl}

We need the following two group-theoretic lemmas which are quite useful in dealing with (generalized) $p$-Jordan groups (see Definitions~\ref{def-p-Jordan} and \ref{def-s-p-Jordan}).
See \cite[Lemmas~2.6 and 2.8]{Popov11} and \cite[Theorem~3]{Popov14} for related results.

\begin{lemma} \label{lemma-product}
Let $G_1$ and $G_2$ be two groups and $G$ their direct product $G_1 \times G_2$.
\begin{enumerate}[{\rm (1)}]
\item If $G_1$ and $G_2$ are generalized $p$-Jordan, then so is $G_1 \times G_2$ and one can take $J(G_1 \times G_2) = J(G_1) J(G_2)$.
\item If $G_1$ and $G_2$ are $p$-Jordan, then so is $G_1 \times G_2$ and one can take $J'(G_1 \times G_2) = J'(G_1) J'(G_2)$, $e(G_1 \times G_2) = e(G_1) + e(G_2) - 1$.
\end{enumerate}
\end{lemma}

\begin{proof}
(1) It follows directly from \cite[Theorem~3(2)]{Popov14}. For the sake of completeness, we present the proof here.
Let $G$ denote the direct product $G_1 \times G_2$ and $\pi_i \colon G \to G_i$ the projection homomorphism.
Let $\Gamma$ be a finite $p'$-subgroup of $G$.
Then $\Gamma_i \coloneqq \pi_i(\Gamma) \le G_i$ contains a normal abelian subgroup $A_i$ such that
$$[\Gamma_i : A_i] \le J(G_i).$$
The subgroup $\wtilde A_i \coloneqq \pi_i^{-1}(A_i) \cap \Gamma$ is normal in $\Gamma$ and $\Gamma/\wtilde A_i$ is isomorphic to $\Gamma_i/A_i$.
We thus have
$$[\Gamma : \wtilde A_i] = [\Gamma_i : A_i] \le J(G_i).$$
Since $A \coloneqq \wtilde A_1 \cap \wtilde A_2$ is the kernel of the diagonal homomorphism
$$\Gamma \lra \Gamma/\wtilde A_1 \times \Gamma/\wtilde A_2$$
defined by the canonical projections $\Gamma \to \Gamma/\wtilde A_i$, we conclude that
$$[\Gamma : A] \le [\Gamma : \wtilde A_1] \cdot [\Gamma : \wtilde A_2] \le J(G_1) J(G_2).$$
By the construction, $A$ is a subgroup of the abelian group $A_1 \times A_2$, so is abelian.
Hence $A$ is a normal abelian subgroup of $\Gamma$ of index $\le J(G_1) J(G_2)$ as claimed.

(2) We need to modify the above proof appropriately.
More precisely, using the notation there, let $\Gamma$ be a finite subgroup of $G = G_1 \times G_2$.
Then $\Gamma_i \coloneqq \pi_i(\Gamma) \le G_i$ contains a normal abelian $p'$-subgroup $A_i$ such that
$$[\Gamma_i : A_i] \le J'(G_i) \cdot |{(\Gamma_i)_{(p)}}|^{e(G_i)}.$$
Note that $|\Gamma_{(p)}| = |(\Gamma_i)_{(p)}| \cdot |(\wtilde A_i)_{(p)}|$ because $\Gamma/\wtilde A_i \isom \Gamma_i/A_i$ and $A_i$ is a $p'$-subgroup of $\Gamma_i$.
Let $A \coloneqq \wtilde A_1 \cap \wtilde A_2 \le A_1 \times A_2$ as above, which is a normal abelian $p'$-subgroup of $\Gamma$.
It follows that
\begin{align*}
[\Gamma : A] &= |{\Gamma/A}| = \frac{\ |\Gamma / \wtilde A_1| \cdot |\Gamma / \wtilde A_2| \ }{|\Gamma / \wtilde A_1 \wtilde A_2 |} \\
&\le \frac{\ J'(G_1) \cdot |{(\Gamma_1)_{(p)}}|^{e(G_1)} \cdot J'(G_2) \cdot |{(\Gamma_2)_{(p)}}|^{e(G_2)} \ }{|(\Gamma / \wtilde A_1 \wtilde A_2)_{(p)}|} \\
&= J'(G_1) J'(G_2) \frac{\ |{(\Gamma_1)_{(p)}}|^{e(G_1)} \cdot |{(\Gamma_2)_{(p)}}|^{e(G_2)} \cdot |(\wtilde A_1 \wtilde A_2)_{(p)}| \ }{|{\Gamma_{(p)}}|} \\
&= J'(G_1) J'(G_2) \frac{\ |{(\Gamma_1)_{(p)}}|^{e(G_1)} \cdot |{(\Gamma_2)_{(p)}}|^{e(G_2)} \cdot |(\wtilde A_1)_{(p)}| \cdot |(\wtilde A_2)_{(p)}| \ }{|{\Gamma_{(p)}}|} \\
&= J'(G_1) J'(G_2) \cdot |{(\Gamma_1)_{(p)}}|^{e(G_1)-1} \cdot |{(\Gamma_2)_{(p)}}|^{e(G_2)-1} \cdot |{\Gamma_{(p)}}| \\
&\le J'(G_1) J'(G_2) \cdot |{\Gamma_{(p)}}|^{e(G_1) + e(G_2) - 1},
\end{align*}
which proves Lemma \ref{lemma-product}.
\end{proof}

\begin{lemma} \label{lemma-finite-quotient}
Let $G$ be a group and $K$ a finite normal subgroup of $G$.
\begin{enumerate}[{\rm (1)}]
\item Suppose that $G$ is generalized $p$-Jordan and one of the following conditions holds:
\begin{enumerate}[{\rm (i)}]
\item the order of $K$ is not divisible by $p$,
\item $p > 0$ and $K$ has a normal Sylow $p$-subgroup $K_{(p)}$.
\end{enumerate}
Then $G/K$ is generalized $p$-Jordan and one can take $J(G/K) = J(G)$.
\item Suppose that $G$ is $p$-Jordan. Then $G/K$ is $p$-Jordan and one can take $$J'(G/K) = J'(G) \cdot |{K_{(p)}}|^{e(G)}, \quad e(G/K) = e(G).$$
\end{enumerate}
\end{lemma}

\begin{proof}
(1) The first case is easy; see e.g. \cite[Lemma~2.6]{Popov11}.
We now consider the case that $p > 0$ and $K_{(p)}$ is nontrivial.
Let $\Gamma$ be a finite $p'$-subgroup of $G/K$. Let $H$ be the inverse of $\Gamma$ in $G$.
Since $p \nmid |{\Gamma}|$ by the assumption, $K_{(p)}$ is also a Sylow $p$-subgroup of $H$.
It follows from $K \lhd H$ that $K_{(p)}$ is also normal in $H$.
Hence $K_{(p)} = H_{(p)}$ is the normal Sylow $p$-subgroup of $H$.
Namely, we have the following exact sequence:
$$1 \lra K/K_{(p)} \lra H/H_{(p)} \lra \Gamma \lra 1.$$
By the Schur--Zassenhaus theorem (cf.~\cite[Theorem~9.1.2]{GTM80}), there is a complement $K_C$ (resp. $H_C$) of $K_{(p)} = H_{(p)}$ in $K$ (resp. $H$), which satisfies that $K = K_{(p)} \rtimes K_C$ (resp. $H = H_{(p)} \rtimes H_C$).
(Their theorem also states that all complements are conjugate to each other, here we do not need this conjugation result though).
Then we rewrite the above exact sequence as follows:
$$1 \lra K_C \lra H_C \lra \Gamma \lra 1.$$
Note that our $H_C$ now is a finite $p'$-subgroup of $G$.
It follows that there is a normal abelian subgroup $A_{H_C}$ of $H_C$ such that $[H_C : A_{H_C}] \le J(G)$.
Let $A$ be the image of $A_{H_C}$ in $\Gamma$.
Then $[\Gamma : A] \le [H_C : A_{H_C}] \le J(G)$.

(2) Let $\Gamma$ be a finite subgroup of $G/K$ and $H$ the inverse of $\Gamma$ in $G$.
Then $H$ contains a normal abelian $p'$-subgroup $A_H$ of index $\le J'(G) \cdot |{H_{(p)}}|^{e(G)}$.
Let $A$ be the image of $A_H$ in $\Gamma$.
Noting that $|H_{(p)}| = |K_{(p)}| \cdot |\Gamma_{(p)}|$, we thus have
$$[\Gamma : A] \le [H : A_H] \le J'(G) \cdot |{H_{(p)}}|^{e(G)} = J'(G) \cdot |{K_{(p)}}|^{e(G)} \cdot |{\Gamma_{(p)}}|^{e(G)}.$$
This yields the assertion (2) and hence Lemma~\ref{lemma-finite-quotient} follows.
\end{proof}

\subsection{Two algebraic group-theoretic theorems} \label{subsec-agtt}

Let $\bG$ be a connected algebraic group defined over a perfect field $\bk$.
We record the following classical decomposition theorem of algebraic groups.
By the Chevalley's structure theorem, there is a smallest connected normal affine subgroup scheme $\bG_\aff$ of $\bG$ such that the quotient $\bG/\bG_\aff$ is an abelian variety (cf.~\cite[Theorem~2]{Brion17}).
On the other hand, $\bG$ has a smallest connected normal subgroup scheme $\bG_\ant$ such that the quotient $\bG/\bG_\ant$ is affine; moreover, $\bG_\ant$ is smooth and contained in the center $Z(\bG)$ of $\bG$ (cf.~\cite[\S5]{Rosenlicht56}).
We have the following Rosenlicht's decomposition theorem (see e.g. \cite[\S5.1]{Brion17}).

\begin{theorem}[{cf.~\cite[Theorem~5.1.1 and Remark~5.1.2]{Brion17}}] \label{theorem-Rosenlicht}
Keep the above notation and assumptions. The following statements hold.
\begin{enumerate}[{\em (1)}]
\item $\bG = \bG_\aff \cdot \bG_\ant \isom \bG_\aff \times \bG_\ant / (\bG_\aff \cap \bG_\ant)$.
\item $\bG_\aff \cap \bG_\ant$ contains $(\bG_\ant)_\aff$.
\item The quotient $(\bG_\aff \cap \bG_\ant) / (\bG_\ant)_\aff$ is finite.
\item The multiplication map of $\bG$ induces an isogeny
$$m \colon (\bG_\aff \times \bG_\ant) / (\bG_\ant)_\aff \lra \bG,$$
where $(\bG_\ant)_\aff$ is viewed as a subgroup scheme of $\bG_\aff \times \bG_\ant$ via $x \mapsto (x, x^{-1})$.
\end{enumerate}
\end{theorem}

The theorem below is a special case of a theorem due to Lucchini Arteche \cite{LA17} which plays an important role in the proof of our main theorems. 
It could be regarded as an effective version of Brion's theorem on the existence of the quasi-splitness of an extension of algebraic groups with finite quotient (cf.~\cite[Theorem~1.1]{Brion15}).

\begin{theorem}[{cf.~\cite[Theorem~3.2]{LA17}}] \label{theorem-LA}
Let $\bk$ be an algebraically closed field of characteristic $p \ge 0$. Let $\Gamma$ be a smooth finite $\bk$-group of order $n$, and $G$ an arbitrary smooth $\bk$-group.
Then for any group extension $$1 \lra G \lra H \lra \Gamma \lra 1,$$
there exist a finite smooth $\bk$-subgroup $S$ of $G$ and a commutative diagram with exact rows
\[
\xymatrix{
1 \ar[r] & S \ar[r] \ar@{^{(}->}[d]_{}^{} & F \ar@{^{(}->}[d]_{}^{} \ar[r]^{} & \Gamma \ar@{=}[d] \ar[r] & 1 \\
1 \ar[r] & G \ar[r] & H \ar[r]^{} & \Gamma \ar[r] & 1.
}
\]
Moreover, if $G$ is an algebraic torus with rank $r$, then $S$ can be taken as a subgroup of the $n$-torsion subgroup $G[n]$ of $G$.
In particular, the order of $S$ divides $n^r$.
\end{theorem}

%%%%%%%%%%%%%%%%%%%%%%%%%%%%%%%%%%%%%%%%%%%%%%%%%%%%%%%%%%

\section{Proofs of theorems} \label{sec-proof}

%%%%%%%%%%%%%%%%%%%%%%%%%%%%%%%%%%%%%%%%%%%%%%%%%%%%%%%%%%

\noindent
We are eventually interested only in questions concerning algebraic groups or varieties defined over algebraically closed fields.
So from now on, we will assume that $\bk$ is algebraically closed.

\begin{proof}[Proof of Theorem \ref{theorem-alg-gp}]
We may assume that $G$ is smooth (or equivalently, reduced), since we only consider finite subgroups of the group $G(\bk)$ of $\bk$-rational points.
We first consider the case that $G$ is {\it connected}.
In the following Steps~\ref{alggp-step1}-\ref{alggp-step3}, we will show the theorem under this case.

\begin{step} \label{alggp-step1}
Let $G_\aff$ and $G_\ant$ denote the affine part and the anti-affine part of $G$, respectively (see \S\ref{subsec-agtt}).
We claim that $\bG_\aff \times \bG_\ant$ is generalized $p$-Jordan. Indeed, since both $G_\aff$ and $G_\ant$ are generalized $p$-Jordan, the claim follows from Lemma~\ref{lemma-product}(1).
More precisely, let $n$ be the least dimension of a faithful representation of $\bG_\aff$ over $\bk$.
Then there is a constant $J(n)$ which is essentially from Theorem~\ref{theorem-Jordan} such that every finite $p'$-subgroup $\Gamma$ of $\bG_\aff \times \bG_\ant$ contains a normal abelian subgroup of index $\le J(n)$.
\end{step}

\begin{step} \label{alggp-step2}
Let $N$ denote $(\bG_\ant)_\aff$.
Consider the following exact sequence of algebraic groups:
$$1 \lra N \lra \bG_\aff \times \bG_\ant \lra (\bG_\aff \times \bG_\ant) / N \lra 1.$$
We claim that the quotient group $(\bG_\aff \times \bG_\ant) / N$ is generalized $p$-Jordan.
Let $\Gamma$ be a finite $p'$-subgroup of $(\bG_\aff \times \bG_\ant) / N$.
Let $H$ be the inverse (or pullback) of the finite group $\Gamma$ in $\bG_\aff \times \bG_\ant$ (cf.~\cite[Proposition~2.8.3]{Brion17}).
Then according to Theorem~\ref{theorem-LA}, there exist a finite smooth $\bk$-subgroup $S$ of $N$ and a commutative diagram of algebraic groups:
\begin{equation} \label{diagram-LA}
\begin{aligned}
\xymatrix{
1 \ar[r] & S \ar[r] \ar@{^{(}->}[d]_{}^{} & F \ar@{^{(}->}[d]_{}^{} \ar[r]^{} & \Gamma \ar@{=}[d] \ar[r] & 1 \\
1 \ar[r] & N \ar[r] & H \ar[r]^{} & \Gamma \ar[r] & 1.
}
\end{aligned}
\end{equation}

In characteristic zero, there is a normal abelian subgroup $A_F$ of $F$ such that $[F : A_F] \le J(n)$ by Step~\ref{alggp-step1}.
It follows that the image $A$ of $A_F$ in $\Gamma$ is a normal abelian subgroup of $\Gamma$ of index $[\Gamma : A] \le [F : A_F] \le J(n)$.

Now we assume that $p = \ch(\bk) > 0$.
By \cite[Proposition 5.5.1]{Brion17}, any anti-affine group over $\bk$ is a semi-abelian variety.
Thus $N = (\bG_\ant)_\aff$ is an algebraic torus over $\bk$.
In particular, $N$ is commutative and so is $S$.
So $S_{(p)}$ is the unique normal Sylow $p$-subgroup of $S$.
By the assumption on $\Gamma$ that $p\nmid |{\Gamma}|$, $S_{(p)}$ is also a Sylow $p$-subgroup of $F$.
It follows from $S\lhd F$ that $S_{(p)}$ is also normal in $F$.
Then the Schur--Zassenhaus theorem asserts that there are complements $S_C$ and $F_C$ of $S_{(p)} = F_{(p)}$ in $S$ and $F$ respectively; see also the proof of Lemma~\ref{lemma-finite-quotient}(1).
It follows that $F_C$ is a $p'$-subgroup of $\bG_\aff \times \bG_\ant$ and hence contains a normal abelian subgroup $A_{F_C}$ of index $\le J(n)$ by Step~\ref{alggp-step1} as in the previous case; the rest is the same.
\end{step}

\begin{step} \label{alggp-step3}
By Rosenlicht's decomposition Theorem \ref{theorem-Rosenlicht}, we have the following exact sequence of algebraic groups:
\begin{equation} \label{seq-isogeny}
1 \lra K \lra (\bG_\aff \times \bG_\ant) / N \xrightarrow{\ \, m \ \, } \bG \lra 1,
\end{equation}
where $N = (\bG_\ant)_\aff$ as in Step~\ref{alggp-step2} and $K \coloneqq (\bG_\aff \cap \bG_\ant) / N$ is a finite group so that $m$, induced from the multiplication map of $\bG$, is an isogeny.
Note that $G_\ant$ is commutative (cf.~\cite[Proposition~3.3.5]{Brion17}).
Then $K$ is abelian and hence $K_{(p)}$ is the normal Sylow $p$-subgroup of $K$.
Let $\Gamma$ be a finite $p'$-subgroup of $\bG$.
It follows from Lemma~\ref{lemma-finite-quotient}(1) and Step~\ref{alggp-step2} that $\Gamma$ contains a normal abelian subgroup $A$ of index $\le J(n)$.
\end{step}

\begin{step}
Finally, with the aid of \cite[Lemma~2.11]{Popov11}, we are able to deal with non-connected algebraic groups as well.
Indeed, let $G^{\circ}$ be the neutral component of $G$.
Denote by $\comp$ the order of the group $\pi_0(G) \coloneqq G/G^{\circ}$ of connected components of $G$.
Then any finite $p'$-subgroup $\Gamma$ of $G$ contains a normal abelian subgroup $A$ of index $\le \comp J(n)^{\comp}$, where $n$ is the least dimension of a faithful representation of $(\bG^\circ)_\aff$ over $\bk$.
\end{step}

We finally conclude the proof of Theorem \ref{theorem-alg-gp} by letting $J(G) = \comp J(n)^{\comp}$.
\end{proof}

%\begin{remark}
%In the proof of the theorem, if we can bound $n$ in terms of the dimension of the linear group $(\bG^\circ)_\aff$, then our $J(G)$ turns out to be a uniform bound in terms of $\dim G$. In characteristic zero, this is true by the classification of semisimple algebraic groups, the structure theorem of reductive groups and the Levi decomposition theorem. This would be interesting in its own right.
%\end{remark}

\begin{proof}[Proof of Theorem~\ref{theorem-alg-gp-LP}]
The proof of Theorem~\ref{theorem-alg-gp-LP} basically follows the strategy of the proof of Theorem~\ref{theorem-alg-gp}.
Note, however, that the group $\Gamma$ may be of order divisible by $p$ now.

We first consider the case that $G$ is a {\it connected} algebraic group.
We may assume that $\bk$ is algebraically closed.
In Step~\ref{alggp-step1}, we have that $J'(G_\aff) = J'(n), e(G_\aff) = 3$ by Theorem~\ref{theorem-LP}, and $J'(G_\ant) = 1, e(G_\ant) = 1$ because $G_\ant$ is commutative (cf.~\cite[Proposition~3.3.5]{Brion17}).
Here $n$ is the least dimension of a faithful representation of $\bG_\aff$ over $\bk$ as usual.
Thus by Lemma~\ref{lemma-product}(2), $J'(\bG_\aff \times \bG_\ant) = J'(\bG_\aff) J'(G_\ant) = J'(n)$ and $e(\bG_\aff \times \bG_\ant) = e(G_\aff) + e(G_\ant) -1 = 3$.
In other words, any finite subgroup $\Gamma$ of $\bG_\aff \times \bG_\ant$ contains a normal abelian $p'$-subgroup of index $\le J'(n) \cdot |{\Gamma_{(p)}}|^3$.

Then in Step~\ref{alggp-step2}, we claim that any finite subgroup $\Gamma$ of $(\bG_\aff \times \bG_\ant) / N$ contains a normal abelian $p'$-subgroup of index $\le J'(n) \cdot |{\Gamma_{(p)}}|^{3(r_G + 1)}$, where $r_G$ is the rank of the algebraic torus $N = (\bG_\ant)_\aff$ which is further bounded by the rank of $G_\aff$.
Indeed, we follow the argument as in the proof of Theorem~\ref{theorem-alg-gp} and get the commutative diagram \eqref{diagram-LA}.
By the previous step, $F$ contains a normal abelian $p'$-subgroup $A_F$ of index $\le J'(n) \cdot |{F_{(p)}}|^3$.
Note that Theorem~\ref{theorem-LA} also yields that the order of $S$ divides $|{\Gamma}|^{r_G}$.
In particular, $|{S_{(p)}}|$ divides $|{\Gamma_{(p)}}|^{r_G}$.
Hence the image $A$ of $A_F$ in $\Gamma$ is a normal abelian $p'$-subgroup of $\Gamma$ such that
$$[\Gamma : A] \le [F : A_F] \le J'(n) \cdot |{F_{(p)}}|^3 = J'(n) \cdot \left(|{S_{(p)}}| \cdot |{\Gamma_{(p)}}| \right)^3 \le J'(n) \cdot |{\Gamma_{(p)}}|^{3(r_G + 1)},$$
as claimed.

In Step~\ref{alggp-step3}, we consider the exact sequence \eqref{seq-isogeny}.
Recall that $K \coloneqq (\bG_\aff \cap \bG_\ant) / N$ is a finite group (depending on $G$ canonically).
Let $\Gamma$ be a finite subgroup of $\bG$.
Then it follows from the previous step and Lemma~\ref{lemma-finite-quotient}(2) that $\Gamma$ contains a normal abelian $p'$-subgroup $A$ of index
$$[\Gamma : A] \le J'(n) \cdot |{K_{(p)}}|^{3(r_G + 1)} \cdot |{\Gamma_{(p)}}|^{3(r_G + 1)}.$$

Lastly, we consider the case that $G$ may be non-connected.
As before, let $\pi_0(G) = G/G^{\circ}$ denote the group of connected components of $G$ and $\comp = |\pi_0(G)|$.
Then any finite subgroup $\Gamma$ of $G$ contains a normal abelian $p'$-subgroup $A$ of index at most
$$\comp \left(J'(n) \cdot |{K_{(p)}}|^{3(r_G + 1)} \cdot |{\Gamma_{(p)}}|^{3(r_G + 1)} \right)^{\comp}
= \comp J'(n)^{\comp} \cdot |{K_{(p)}}|^{3(r_G + 1) \comp} \cdot |{\Gamma_{(p)}}|^{3(r_G + 1) \comp},$$
where $n$ is the least dimension of a faithful representation of $(\bG^\circ)_\aff$ over $\bk$.

Let $J'(G)$ denote $\comp J'(n)^{\comp} \cdot |{K_{(p)}}|^{3(r_G + 1) \comp}$ and $e(G) \coloneqq 3(r_G + 1) \comp$.
We thus complete the proof of Theorem~\ref{theorem-alg-gp-LP}.
\end{proof}

Given a projective variety $X$ defined over $\bk$, the automorphism group scheme $\aut_X$ of $X$ is locally of finite type over $\bk$ and the automorphism group $\Aut(X)$ is just the rational $\bk$-points of $\aut_X$, i.e., $\Aut(X) = \aut_X(\bk)$; in particular, the reduced neutral component $(\aut^{\circ}_X)_{\red}$ of $\aut_X$ is a smooth algebraic group defined over $\bk$ (see e.g. \cite[\S7]{Brion17}).
We denote $(\aut^{\circ}_X)_{\red}(\bk)$ by $\Aut^{\circ}(X)$.

\begin{proof}[Proof of Theorem \ref{thm-aut}]
Let $G \coloneqq (\aut^{\circ}_X)_{\red}$ and $\Gamma^{\circ} \coloneqq \Gamma\cap \Aut^{\circ}(X)$.
Applying Theorem~\ref{theorem-alg-gp} to $\Gamma^{\circ} \le \bG(\bk)$, there is a normal abelian subgroup $A_{\Gamma^{\circ}}$ of $\Gamma^{\circ}$ of index $\le J(G) = J(n)$, where $n$ is the least dimension of a faithful representation of $G_\aff$ over $\bk$.
By \cite[Lemma~2.5]{MZ18}, there is a constant $\ell_X$ depending only on $X$ such that $[\Gamma : \Gamma^{\circ}] \le \ell_X$ (note that their proof is independent of the characteristic; see \cite[Remark~2.6]{MZ18}).
%Thus $A_{\Gamma^{\circ}}$ is an abelian subgroup of $\Gamma$ of index $\le \ell_X J(n)$.
%It follows from \cite[Theorem~1.41]{Isaacs08} that $\Gamma$ has a characteristic (and hence normal) abelian subgroup $A$ such that
%$$[\Gamma : A] \le [\Gamma : A_{\Gamma^{\circ}}]^2 \le \ell_X^2 J(n)^2.$$
%To conclude the proof, we let $J_X \coloneqq \ell_X^2 J(n)^2$.
It follows that
$$A \coloneqq \bigcap_{g \in \Gamma} g^{-1} A_{\Gamma^{\circ}} g = \bigcap_{i=1}^{|\Gamma/\Gamma^{\circ}|} g_i^{-1} A_{\Gamma^{\circ}} g_i$$
is a normal abelian subgroup of $\Gamma$, where $g_i$'s are representatives of $\Gamma/\Gamma^{\circ}$.
Note that $g_i^{-1} A_{\Gamma^{\circ}} g_i$ is a normal abelian group of $\Gamma^{\circ}$ of index $\le J(n)$ for each $i$.
This yields that the index of $A$ in $\Gamma^{\circ}$ is at most $J(n)^{\ell_X}$ (see also \cite[Lemma~2.11]{Popov11}).
We thus have
$$[\Gamma : A] = [\Gamma : \Gamma^{\circ}] \cdot [\Gamma^{\circ} : A] \le \ell_X J(n)^{\ell_X}.$$
To conclude the proof, we let $J_X \coloneqq \ell_X J(n)^{\ell_X}$.
\end{proof}

\begin{remark}
In fact, using a theorem due to Chermak and Delgado (cf.~\cite[Theorem~1.41]{Isaacs08}), there exists a characteristic (and hence normal) abelian subgroup $M$ of $\Gamma$ such that
$$[\Gamma : M] \le [\Gamma : A_{\Gamma^{\circ}}]^2 = [\Gamma : \Gamma^{\circ}]^2 \cdot [\Gamma^{\circ} : A_{\Gamma^{\circ}}]^2 \le \ell_X^2 J(n)^2.$$
%Then one may take $J_X$ as $\ell_X^2 J(n)^2$.
Note, however, that the construction of this so-called Chermak--Delgado subgroup $M$ of $\Gamma$ is independent of $A_{\Gamma^{\circ}}$ (see \cite[Corollary~1.45]{Isaacs08}).
Thus we do not know whether $M$ is still a subgroup of $A_{\Gamma^{\circ}}$ so that this argument breaks down in the proof of Theorem~\ref{thm-aut-LP} (since $M$ may not be a $p'$-subgroup of an arbitrary finite subgroup $\Gamma$ of $G(\bk)$).
\end{remark}

\begin{proof}[Proof of Theorem \ref{thm-aut-LP}]
Let $G \coloneqq (\aut^{\circ}_X)_{\red}$ and $\Gamma^{\circ} \coloneqq \Gamma\cap \Aut^{\circ}(X)$ as above.
Then applying Theorem~\ref{theorem-alg-gp-LP} to $\Gamma^{\circ} \le \bG(\bk)$, there is a normal abelian $p'$-subgroup $A_{\Gamma^{\circ}}$ of $\Gamma^{\circ}$ such that
$$[\Gamma^{\circ} : A_{\Gamma^{\circ}}] \le J'(G) \cdot |\Gamma^{\circ}_{(p)}|^{3(r_G + 1)},$$
where $\Gamma^{\circ}_{(p)}$ is the Sylow $p$-subgroup of $\Gamma^{\circ}$ and $r_G$ is the rank of $(\bG_\ant)_\aff$.
As in the proof of Theorem~\ref{thm-aut}, we also have $[\Gamma : \Gamma^{\circ}] \le \ell_X$ for some constant $\ell_X$ depending only on $X$.
Similarly, there is a normal abelian $p'$-subgroup $A$ of $\Gamma$ of index
$$[\Gamma : A] = [\Gamma : \Gamma^{\circ}] \cdot [\Gamma^{\circ} : A] \le \ell_X \left(J'(G) \cdot |\Gamma^{\circ}_{(p)}|^{3(r_G + 1)}\right)^{\ell_X} \le \ell_X J'(G)^{\ell_X} \cdot |{\Gamma_{(p)}}|^{3(r_G + 1)\ell_X}.$$
The corollary follows by letting $J'_X \coloneqq \ell_X J'(G)^{\ell_X}$ and $e_X = 3(r_G + 1)\ell_X$.
\end{proof}

\begin{remark}
It is known that if an algebraic torus $T$ acting faithfully on an algebraic variety $X$, then $T$ acts generically freely on $X$ (cf.~\cite[\S1.6, Corollaire~1]{Demazure70}).
This yields that
$$r_G = \rank (\bG_\ant)_\aff \le \rank G_\aff \le \dim X.$$
\end{remark}

%%%%%%%%%%%%%%%%%%%%%%%%%%%%%%%%%%%%%%%%%%%%%%%%%%%%%%%%%%

\phantomsection
\addcontentsline{toc}{section}{Acknowledgments}
\noindent \textbf{Acknowledgments. }
The author would like to thank Zinovy Reichstein for his support, many inspiring discussions and valuable comments on an earlier draft.
He also thanks De-Qi Zhang for several helpful suggestions.
He is grateful to an anonymous referee for pointing out an inaccuracy in a former version.

%%%%%%%%%%%%%%%%%%%%%%%%%%%%%%%%%%%%%%%%%%%%%%%%%%%%%%%%%%

\linespread{1.1}

\bibliographystyle{amsalpha}
%\bibliographystyle{siam}
%\bibliographystyle{alpha}
%\bibliographystyle{plain}
%\bibliographystyle{unsrt}
%\bibliographystyle{abbrv}
%\bibliography{../mybib}

\providecommand{\bysame}{\leavevmode\hbox to3em{\hrulefill}\thinspace}
\providecommand{\MR}{\relax\ifhmode\unskip\space\fi MR }
\providecommand{\MRhref}[2]{%
  \href{http://www.ams.org/mathscinet-getitem?mr=#1}{#2}
}
\providecommand{\href}[2]{#2}

\end{document}